\documentclass{article}

\usepackage{multicol}
\usepackage{multirow}
\usepackage{fullpage}
\usepackage{amsthm}
\usepackage{latexsym}
\usepackage{amsmath,amscd}
\usepackage{amsfonts}
\usepackage{amssymb}
\usepackage[T1]{fontenc}
\usepackage{verbatim} 
\usepackage[all]{xy}
\usepackage[french,english]{babel}
\usepackage{blkarray}

\newcommand{\bP}{\mathbb{P}}
\newcommand{\bC}{\mathbb{C}}
\newcommand{\cO}{\mathcal{O}}

\newcommand{\Om}{\Omega}
\newcommand{\tOm}{\widetilde{\Omega}}

\DeclareMathOperator{\codim}{codim}

\DeclareMathOperator{\Gr}{Gr}

\theoremstyle{plain}
\newtheorem{theorem}{Theorem}[section]
\newtheorem{lemma}[theorem]{Lemma}
\newtheorem{e-proposition}[theorem]{Proposition}
\newtheorem{corollary}[theorem]{Corollary}

\newtheorem{theointro}{Theorem}

\theoremstyle{definition}
\newtheorem{e-definition}[theorem]{Definition\rm}
\theoremstyle{remark}
\newtheorem{remark}[theorem]{\it Remark\/}

\begin{document}

\title{Differential equations as embedding obstructions and vanishing theorems
}


\author{Damian Brotbek     
}

\maketitle

\begin{abstract}
We generalize a vanishing theorem for the cohomology of symmetric powers of the cotangent bundle of subvarieties of projective space due to Schneider. From this we deduce new vanishing results for Green-Griffiths jet differential bundles, generalizing results of Diverio and Pacienza-Rousseau.
\end{abstract}

\section{Introduction}
\label{intro}
Given a smooth projective $n$-dimensional variety $X$, one can search for the smallest integer $N$ such that there exists an embedding $X\hookrightarrow \bP^N$. Obviously, $n\leqslant N$, and, on the other hand, a simple argument shows that $N\leqslant 2n+1$ (see $\cite{Har74}$ \textsection  4). In general, these bounds are optimal, and therefore it seems natural to try to understand what geometric properties of $X$ yield obstructions  for projective embeddings. Schneider proved in $\cite{Sch92}$ that the existence of symmetric differential forms on $X$ imposes strong restriction on such embeddings. His result, and its generalization for higher order cohomology groups can be stated as follows.

\begin{theorem}[Schneider \cite{Sch92} Theorems 1.1 and 1.2]
Let $X \subset \bP^N$ be a  smooth subvariety of dimension $n$ and codimension $c$. Let $m\in \mathbb{N}$ and $a\in \mathbb{Z}$. Consider an integer $0\leqslant j<n-c$. If $a< m-\min\{j,1\}$, then
$$H^j(X,S^m\Om_X\otimes \cO_{X}(a))=0.$$
\end{theorem}
If one regards global sections of $S^k\Om_X$ as algebraic differential equations of order $1$ defined on $X$, one is naturally led to consider the problem of finding a higher order analogue to this statement. In that direction, Diverio  $\cite{Div08}$ established a vanishing theorem for Green-Griffiths jet differential bundles on complete intersection varieties (see section \ref{prelimGGjets} for the definition). 
\begin{theorem}[Diverio \cite{Div08} Theorem 7]\label{DiverioVanishing}
Let $X\subseteq \bP^N$ be a smooth complete intersection of dimension $n$ and codimension $c$. Then
$$H^0(X,E^{GG}_{k,m}\Om_X)=0$$
for all $m\geqslant 1$ and $1\leqslant k < \frac{n}{c}$.
\end{theorem}
Pacienza and Rousseau $\cite{P-R08}$ constructed a generalized version of Green-Griffiths jet differential bundle and they proved a similar result for those more general bundles (see Theorem $\ref{P-RVanishing}$).\\ 
The aim of our work is to generalize, and unify, Schneider's and Diverio's result. Our main result is the following (see Theorem \ref{SchneiderOm}).
\begin{theointro}
Let $X\subseteq \bP^N$ be a smooth variety of dimension $n$ and codimension $c=N-n$. Consider an integer $k\geqslant 1$, integers $\ell_1,\dots, \ell_k\geqslant 0$, and $a\in \mathbb{Z}$. If $j< n-k\cdot c$ and $a<\ell_1+\cdots +\ell_k-\min\{j,k\}$, then 
$$H^j(X,S^{\ell_1}\Om_X\otimes \cdots \otimes S^{\ell_k}\Om_X\otimes \cO_{X}(a))=0.$$
\end{theointro}

From this generalization of Schneider's theorem one can deduce a new vanishing result for jet differentials (see Corollary $\ref{VanishingGGjets}$).  

\begin{theointro}
Let $X\subseteq \bP^N$ be a smooth variety of dimension $n$ and codimension $c$. Let $a\in \mathbb{Z}$.
If $j< n-k\cdot c$ and $a<\frac{m}{k}-\min\{j,k\}$ then
$$H^j(X,E^{GG}_{k,m}\Om_X\otimes \cO_X(a))=0.$$
\end{theointro}

Therefore the existence of higher order jet differential equation is an obstruction to projective embeddings, analogues to the ones pointed out by Schneider. We also prove an analogous result for Pacienza-Rousseau's generalized Green-Griffiths jet differential bundles (see Corollary \ref{VanishingPRGGjets}).\\

\emph{Notation and conventions.}
 We work over the field  of complex numbers $\bC$. When $X$ is a complex manifold, we denote the tangent bundle of $X$ by $TX$, and the cotangent bundle of $X$ by $\Om_X:=TX^*$.

\section{The bundle $\tOm$}
It will be convenient for us to work with the bundle $\tOm$, studied in particular  by  Bogomolov and DeOliveira in $\cite{BdO08}$, but also by Debarre in $\cite{Deb05}$. Therefore we recall some basic fact on this bundle. When $X\subseteq \bP^N$ is an $n$-dimensional variety then $\tOm_X$ is, roughly speaking, the sheaf of differential forms on $\widehat{X}\subseteq \bC^{N+1}$, the cone over $X$, which are invariant under the $\bC^*$-action but which do not necessarily satisfy the Euler condition. A more geometrical approach is to consider the Gauss map
\begin{eqnarray*}
\gamma_X : X &\to& \Gr(n,\bP^N)\\
x&\mapsto& \mathbb{T}_xX
\end{eqnarray*}
where $\mathbb{T}_xX\subseteq \bP^N$ is the embedded tangent space of $X$ at $x$, and where $\Gr(n,\bP^N)$ denotes the Grassmannian variety of $n$-dimensional projective subspaces of $\bP^N$. Let $\mathcal{S}_n$ denote the rank-$n$ tautological bundle on $\Gr(n,\bP^N)$, then set
$$\tOm_X=\gamma_X^*\mathcal{S}_n^*\otimes \cO_X(-1).$$
Observe that $\tOm_{\bP(V)}=V\otimes \cO_{\bP(V)}(-1)$. The properties of $\tOm_X$ we will use are summarized in the following commutative diagramm.

$$\begin{CD}
  @.            0                @.                 0          @.                                 @.          \\
@.            @VVV                                @VVV                          @.                          @.\\
  @.         N^*_{X/\bP^N}            @=          N^*_{X/\bP^N}        @.                                 @.          \\   
@.            @VVV                                @VVV                          @.                          @.\\ 
0 @>>> \Om_{\bP^N_{|X}} @>>>    \tOm_{\bP^N_{|X}}       @>>>       \cO_X  @>>>    0\\  
@.            @VVV                                @VVV                          @|                          @.\\ 
0 @>>>\Om_{X}  @>>>    \tOm_{X}      @>>>       \cO_X   @>>>     0 \\
@.            @VVV                                @VVV                          @.                          @.\\
  @.           0                @.                 0            @.                                 @.         \\ \\
\end{CD}$$
We refer to $\cite{BdO08}$ and $\cite{Deb05}$ for some details. In this paper the exact sequence
$$0\to N^*_{X/\bP^N}\to \tOm_{\bP^N_{|X}}\to \tOm_X\to 0$$
will be called the \emph{tilde conormal exact sequence}. And the exact sequence 
$$0\to \Om_X \to \tOm_X \to \cO_X\to 0$$
will be called the \emph{Euler exact sequence}. 

\section{The Green-Griffiths jet differential bundles $E^{GG}_{k,m}\Om_X$}\label{prelimGGjets}
We recall very briefly the definitions of Green-Griffiths jet differential bundles. We refer to $\cite{Dem97}$ for more information, and also to $\cite{Mer10}$ where many details are carried out explicitly.
Let $X$ be a projective variety of dimension $n$. For all $k\geqslant 1$, we denote by $J_kX\stackrel{p_k}{\to} X$ the holomorphic bundle of $k$-jets of germs of holomorphic curves $f: (\bC,0)\to X$, that is  $J_kX:=\{f: (\bC,0)\to X\}/\sim$, where two germs $f,g:(\bC,0)\to X$ are equivalent $(f\sim g)$ if and only if $f^{(j)}(0)=g^{(j)}(0)$ for all $0\leqslant j\leqslant k$.
The projection is then simply defined by
\begin{eqnarray*}
J_kX&\stackrel{p_k}{\rightarrow}& X\\
f&\mapsto& f(0). 
\end{eqnarray*}


Those spaces naturally have the structure of holomorphic fiber bundles over $X$, but unless $k=1$ they are not vector bundles. Moreover, for any $x\in X$, $J_kX_x\cong \left(\bC^n\right)^k$. To see this, take local coordinates $(z_1,\dots , z_n)$ around $x\in X$; we can then write $f=(f_1,\dots,f_n)$ and the $k$-jet will then be entirely determined, by Taylor's formula, by
$$(f'_1(0),\dots , f'_n(0),f''_1(0),\dots,f''_n(0), \dots ,f^{(k)}_1(0),\dots, f^{(k)}_n(0)).$$

With the above notation, for each $k\geqslant 1$, there is a natural $\bC^*$-action on $J_kX$. Namely if $\lambda\in\bC^*$ and $f:(\bC,0)\to (X,x)$ then 
\begin{eqnarray*}
\lambda\cdot f : (\bC,0)&\to& (X,x)\\
t&\mapsto& f(\lambda t).
\end{eqnarray*}
 This action is easily expressed fiberwise. Namely, if one considers local coordinates around $x=f(0)=\lambda\cdot f(0)$ then $f$ is represented by 
$(f'(0),f''(0),\dots, f^{(k)}(0))$
and 
$$\lambda\cdot \left(f'(0),f''(0),\dots, f^{(k)}(0)\right)=\left(\lambda f'(0),\lambda^2 f''(0),\dots, \lambda^k f^{(k)}(0)\right).$$
Then the Green-Griffiths jet differentials are defined as follows. Fiberwise we consider
\begin{eqnarray*}
E_{k,m}^{GG}\Om_{X,x}:=\left\{Q\in \bC[f',\dots,f^{(k)}]\ / \ Q\left(\lambda\cdot \left(f',\dots, f^{(k)}\right)\right)=\lambda^mQ\left(f',\dots, f^{(k)}\right)\right\}.
\end{eqnarray*}
One can make this even more explicit: when $x\in X$ is fixed, one can consider coordinates $$(f_1',\dots, f_n', \dots, f_1^{(k)},\dots, f_n^{(k)})$$ 
on $J_kX_x$. Then an element $Q\in E_{k,m}^{GG}\Om_{X,x}$ is exactly a polynomial in the variables $(f_i^{(j)})$ of the form
$$Q=\sum_{\substack{I_1,\dots,I_k\in \mathbb{N}^n \\ |I_1|+2|I_2|+\cdots +k|I_k|=m }} a_{I_1,\dots , I_k}(f')^{I_1}\cdots (f^{(k)})^{I_k}$$
where we use the standard multi-index notation for $I=(i_1,\dots , i_n)$ and $1\leqslant j \leqslant k$ we set $\left(f^{(j)}\right)^{I}:=\left(f_{1}^{(j)}\right)^{i_1}\cdots \left(f_{n}^{(j)}\right)^{i_n}$. It turns out that these fibers can be arranged into a vector bundle over $X$.\\
The bundle $E_{k,m}^{GG}\Om_X$ admits a natural filtration. We briefly recall its construction; however, we will not go into details, in particular we don't justify why everything is well-defined. \\
Fix local coordinates around $x\in X$ as above. For each $p\in \mathbb{N}$ and for each $1\leqslant s\leqslant k$ define
\begin{eqnarray*}
F_s^p=F_s^p(E_{k,m}^{GG}\Om_{X,x}) =\left\{ 
\begin{array}{c}
Q\in E_{k,m}^{GG}\Om_{X,x}\ \text{involving only monomials}\ (f')^{I_1}\cdots (f^{(k)})^{I_k}\\
 \text{with}\ |I_1|+2|I_2|+\cdots +s|I_s|\geqslant p
\end{array}
\right\}.
\end{eqnarray*}
This gives a filtration 
$$\{0\}=F_s^{m+1}\subseteq F_s^{m}\subseteq\cdots\subseteq F^1_s\subseteq F^0_s=E_{k,m}^{GG}\Om_X.$$
We consider the associated graded terms 
$$\Gr_{s}^p=\Gr_s^p(E_{k,m}^{GG}\Om_X):=F^{p}_s/F^{p+1}_s.$$
Observe that 
\begin{eqnarray*}
\Gr_{k-1}^p&\cong& \left\{\begin{array}{c}
Q\ \text{involving only monomials}\ (f')^{I_1}\cdots (f^{(k)})^{I_k}\\
 \text{with}\ |I_1|+2|I_2|+\cdots +(k-1)|I_{k-1}|=p 
\end{array}\right\}\\
&\cong& \left\{\begin{array}{c}
Q\ \text{involving only monomials}\ (f')^{I_1}\cdots (f^{(k)})^{I_k}\\
 \text{with}\ k|I_{k}|=m-p 
\end{array}\right\}.
\end{eqnarray*}
Therefore $\Gr_{k-1}^p\neq 0$ if and only if there is an integer $\ell_k\in \mathbb{N}$ such that $p=m-k\ell_k$. Whenever this is satisfied, if one looks closely at the coordinate changes, one observe that 
$$\Gr_{k-1}^p=\Gr_{k-1}^{m-k\ell_k}\cong E_{k-1,m-k\ell_k}^{GG}\Om_X\otimes S^{\ell_k}\Om_X.$$
And therefore,
$$\Gr_{k-1}^{\bullet}=\bigoplus_{1\leq p\leq m}\Gr_{k-1}^p\cong \bigoplus_{0\leqslant \ell_k\leqslant \lfloor \frac{m}{k}\rfloor} E_{k-1,m-k\ell_k}^{GG}\Om_X\otimes S^{\ell_k}\Om_X.$$
Combining all those filtrations, we find inductively a filtration $F^{\bullet}$ on $E_{k,m}^{GG}\Om_X$ such that the associated graded bundle is

$$\Gr^{\bullet}(E_{k,m}^{GG}\Om_X)=\bigoplus_{\ell_1+2\ell_2+\cdots +k\ell_k=m}S^{\ell_1}\Om_X\otimes\cdots \otimes S^{\ell_k}\Om_X.$$

  \section{Pacienza-Rousseau generalized jet differential bundles}

The Green-Griffiths jet differential bundles were constructed to study entire maps $f:\bC\to X$. Pacienza and Rousseau $\cite{P-R08}$ generalized this construction to study more generally holomorphic maps $f:\bC^p\to X$. For each $p\geq 1$ they constructed bundles $E_{p,k,m}^{GG}\Om_X$ generalizing $E_{k,m}^{GG}\Om_X=E_{1,k,m}^{GG}\Om_X$. Their construction is very similar to the one described above, however there are some unexpected difficulties appearing. We just really briefly recall the definitions. For the details we refer to \cite{P-R08}.\\

Fix $p\geq1$. Consider the space $J_{k,p}X$ of $k$-jets of germs of holomorphic maps $f:(\bC^p,0)\to X$. As above, this space comes with a natural $(\bC^*)^p$-action. Namely, for $f:(\bC^p,0)\to X$ and $\lambda=(\lambda_1,\dots , \lambda_p)\in(\bC^*)^p$ we define
\begin{eqnarray*}
\lambda\cdot f:(\bC^p,0)&\to& X\\
(t_1,\dots, t_p)&\mapsto& f(\lambda_1t_1,\dots, \lambda_pt_p).
\end{eqnarray*}
The bundles $E_{p,k,m}^{GG}\Om_X$ are defined as follows: for each $x\in X$ we consider
\begin{eqnarray*}
E_{p,k,m}^{GG}\Om_{X,x}:=\left\{\begin{array}{c}Q(f',\dots, f^{(k)}) \ / \ Q(\lambda\cdot (f',\dots, f^{(k)})=\lambda_1^m\cdots \lambda_p^mQ(f',\dots, f^{(k)})\\ \text{for all}\ \lambda=(\lambda_1,\dots , \lambda_p)\in(\bC^*)^p\end{array}\right\}.
\end{eqnarray*}
 Those spaces can be arranged into vector bundles. The only thing we will need concerning these bundles is that $E^{GG}_{p,k,m}\Om_X$ admits a filtration whose graded terms are 
$$\bigotimes_{\alpha\in I_1}S^{q_{\alpha}^1}\Om_X\cdots \bigotimes_{\alpha\in I_k}S^{q_{\alpha}^k}\Om_X,$$
where 
$$\sum_{\ell=1}^k\sum_{\alpha\in I_{\ell}}q_{\alpha}^{\ell}\alpha=(m,\dots,m),$$
and where for $\ell\in\mathbb{N}$, $I_{\ell}:=\{\alpha=(\alpha_1,\dots, \alpha_p)\in \mathbb{N}^p\ /\ \alpha_1+\cdots +\alpha_p=\ell\}$. See Remark 2.6 in  $\cite{P-R08}$.

\section{A cohomological lemma}

During the proof of our main result, we will need an elementary cohomological lemma.
\begin{lemma}\label{lemmedroite}
Let $X$ be a projective variety. Let $G$ be a vector bundle on $X$. Let $k\geqslant 1$. Suppose we have $k$ long exact sequences of vector bundles on $X$:
\begin{eqnarray*}
0\to E_1^{\ell}\to E_1^{{\ell}-1}\to &\cdots & \to E_1^{1}\to E_1^0 \to F_1 \to 0 \\
0\to E_2^{\ell}\to E_2^{{\ell}-1}\to &\cdots & \to E_2^{1}\to E_2^0 \to F_2 \to 0 \\
&\vdots & \\
0\to E_k^{\ell}\to E_k^{{\ell}-1}\to &\cdots & \to E_k^{1}\to E_k^0 \to F_k \to 0. 
\end{eqnarray*}
Fix an integer $q$. Suppose that 
$$H^j(X,E_1^{i_1}\otimes \cdots \otimes E_k^{i_k}\otimes G)=0 \ \ for \ all \ \ j\leqslant q+i_1+\cdots +i_k.$$
Then 
$$H^j(X,F_1\otimes \cdots \otimes F_k\otimes G)=0  \ \ for\ all \ \ j\leqslant q.$$
\end{lemma}

\begin{remark}
The case $k=1$ appears already in Schneider's article \cite{Sch92}.
\end{remark}

\begin{proof}
The proof is straightforward, but we write it down for the sake of completeness. This is just an induction on $k$.\\

Let us start when $k=1$. We proceed by induction on ${\ell}$.\\
When ${\ell}=0$, this is obvious. When ${\ell}=1$, we just have one short exact sequence
$$0\to E_1^1\to E_1^0\to F_1\to 0.$$
Tensoring it by $G$ and looking at the associated long exact sequence in cohomology gives the result. Now when ${\ell}\geqslant 2$, we cut the long exact sequence into two pieces, and tensor everything by $G$ to obtain
\begin{eqnarray*}
0\to E_1^{\ell}\otimes G\to E_1^{{\ell}-1}\otimes G\to &K\otimes G& \to 0 \\
0\to&K\otimes G&\to E_1^{{\ell}-2}\otimes G \to \cdots \to E_1^1\otimes G\to E_1^0\otimes G\to F_1\otimes G \to 0.
\end{eqnarray*}
Apply the global section functor to the first exact sequence to obtain $H^j(X,K\otimes G)=0$ for all $j\leqslant q+{\ell}-1$, and then apply the induction hypothesis.\\
  
Now we let $k\geqslant 2$.
Suppose the result holds for any family of $k$ exact sequences.
Take one more exact sequence 
\begin{eqnarray}
0\to E_{k+1}^{\ell}\to E_{k+1}^{{\ell}-1}\to \cdots \to E_{k+1}^1\to E_{k+1}^0\to F_{k+1}\to 0.\label{exactsequence-c+1}
\end{eqnarray}
Suppose that $H^j(X,E_1^{i_1}\otimes \cdots \otimes E_{k+1}^{i_{k+1}}\otimes G)=0$ for all $j\leqslant q+i_1+\cdots + i_{k+1}$.

 Tensoring the exact sequence (\ref{exactsequence-c+1}) by $F_1\otimes \cdots \otimes F_k\otimes G$, we obtain
 
 $$0\to \widetilde{E}_{k+1}^{\ell}\to \widetilde{E}_{k+1}^{{\ell}-1}\to \cdots \to \widetilde{E}_{k+1}^1\to \widetilde{E}_{k+1}^0 \to F_1\otimes \cdots \otimes F_{k+1}\otimes G\to 0,$$
 where $\widetilde{E}_{k+1}^i:=E_{k+1}^i\otimes F_1 \otimes \cdots \otimes F_k\otimes G$.
Therefore, to prove that $H^j(X,F_1\otimes \cdots \otimes F_{k+1}\otimes G)=0$ for all $j\leqslant q$, it suffices to prove that for any $0\leqslant i \leqslant {\ell}$,
$$H^j(X,\widetilde{E}_{k+1}^{i})=0$$
 for all $j \leqslant q+i$. To do so, fix $0\leqslant i\leqslant {\ell}$. Consider the $k$ long exact sequences
\begin{eqnarray*}
0\to E_1^{\ell}\otimes E_{k+1}^i\to E_1^{{\ell}-1}\otimes E_{k+1}^i\to &\cdots & \to E_1^{1}\otimes E_{k+1}^i\to E_1^0 \otimes E_{k+1}^i\to F_1\otimes E_{k+1}^i \to 0 \\
0\to E_2^{\ell}\to E_2^{{\ell}-1}\to &\cdots & \to E_2^{1}\to E_2^0 \to F_2 \to 0 \\
&\vdots & \\
0\to E_k^{\ell}\to E_k^{{\ell}-1}\to &\cdots & \to E_k^{1}\to E_k^0 \to F_k \to 0. \\
\end{eqnarray*}
By hypothesis,  $H^j(X,E_1^{i_1}\otimes E_{k+1}^i\otimes E_2^{i_2}\otimes \cdots \otimes E_k^{i_k}\otimes G)=0$ for all $j\leqslant q+ i_1+\cdots + i_k + i$. Therefore by the induction hypothesis, we obtain that $H^j(X,\widetilde{E}_{k+1}^i)=H^j(X,E_{k+1}^i\otimes F_1 \otimes \cdots \otimes F_k\otimes G)=0$ for all $j\leqslant q+i$. This concludes the proof.
\end{proof}

Similarly, we obtain the following.

\begin{lemma}\label{lemmegauche}
Let $X$ be a projective variety. Let $k\geqslant 1$. Suppose we have $k$ long exact sequences of vector bundles on $X$,
\begin{eqnarray*}
0\to F_1\to E_1^{0}\to E_1^1 &\cdots & \to E_1^{{\ell}-1}\to E_1^{\ell} \to 0 \\
0\to F_2\to E_2^{0}\to E_2^1 &\cdots & \to E_2^{{\ell}-1}\to E_2^{\ell} \to 0 \\
&\vdots & \\
0\to F_k\to E_k^{0}\to E_k^1 &\cdots & \to E_k^{{\ell}-1}\to E_k^{\ell} \to 0. 
\end{eqnarray*}
Fix an integer $q$. Suppose that 
$$H^j(X,E_1^{i_1}\otimes \cdots \otimes E_k^{i_k}\otimes G)=0 \ \ for\ all \ \ j\leqslant q-i_1-\cdots -i_k.$$
Then 
$$H^j(X,F_1\otimes \cdots \otimes F_k\otimes G)=0  \ \ for\ all \ \ j\leqslant q.$$
\end{lemma}


\section{Main results}

To prove his result, Schneider considered a Kozsul resolution of $S^k\Om_X$. Then, applying his cohomological lemma with Le Potier's vanishing theorem \cite{LeP75}, he was able to conclude.
Here we follow the same idea in our more general setting. However, we will not be able to apply directly the vanishing theorem of Le Potier, we have to replace it with a more general vanishing theorem. The following theorem appears first in a work of Ein and Lasarsfeld \cite{E-L93}, but this result follows from Le Potier's theorem thanks to an idea of Manivel. We refer to \cite{E-L93} and \cite{Man97} for more details.

\begin{theorem}[Ein-Lazarsfeld/Manivel]\label{Manivel}
Let $X$ be a smooth projective variety of dimension $n$. Let $E_1,\dots, E_r$ be vector bundles on $X$, of ranks $e_1,\dots,e_r$, and let $A$ be an ample line bundle on $X$. Assume that each $E_i$ is generated by global sections, and fix $a_1,\dots, a_r\geqslant 1$. Then
$$H^k(X,K_X\otimes \Lambda^{a_i}E_1\otimes \cdots \otimes \Lambda^{a_r}E_r\otimes A)=0$$
for $k>(e_1-a_1)+\cdots +(e_r-a_r)$.
\end{theorem}
 
The case $r=1$ is the original result of Le Potier \cite{LeP75}.
 
Before we continue we recall a well-known fact. Consider a subvariety $X\subseteq \bP^N$, and let $N_{X/\bP^N}$ denote the normal bundle to $X$ in $\bP^N$. Then $N_{X/\bP^N}\otimes \cO_{X}(-1)$ is globally generated. This follows directly from the normal exact sequence and the Euler exact sequence. As a matter of notation, we fix an $(N+1)$-dimensional complex vector space such that $\bP^N=\bP(V)$.  \\

Putting all this together, we can prove the following.

\begin{theorem}\label{SchneidertOm}
Let $X\subseteq \bP^N$ be a smooth variety of dimension $n$ and codimension $c=N-n$. Consider an integer $k\geqslant 1$, integers $\ell_1,\dots, \ell_k\geqslant 0$, and $a\in \mathbb{Z}$. If $j< n-k\cdot c$ and $a<\ell_1+\cdots +\ell_k$, then 
$$H^j(X,S^{\ell_1}\tOm_X\otimes \cdots \otimes S^{\ell_k}\tOm_X\otimes \cO_{X}(a))=0.$$
\end{theorem}
\begin{proof}
Let $N:=N_{X/\bP^N}$. Recall the tilde conormal exact sequence
$$0\to N^*\to \tOm_{\bP^N_{|X}}\to \tOm_X\to 0.$$
Taking different symmetric powers we obtain $k$ exact sequences 
\begin{eqnarray*}
0\to \Lambda^{\ell_1}N^*\to \Lambda^{\ell_1-1}N^*\otimes \tOm_{\bP^N_{|X}}\to &\cdots& \to N^*\otimes S^{\ell_1-1}\tOm_{\bP^N_{|X}} \to S^{\ell_1}\tOm_{\bP^N_{|X}}\to S^{\ell_1}\tOm_X\to 0\\
0\to \Lambda^{\ell_2}N^*\to \Lambda^{\ell_2-1}N^*\otimes \tOm_{\bP^N_{|X}}\to &\cdots& \to N^*\otimes S^{\ell_2-1}\tOm_{\bP^N_{|X}} \to S^{\ell_2}\tOm_{\bP^N_{|X}}\to S^{\ell_2}\tOm_X\to 0\\
&\vdots&\\
0\to \Lambda^{\ell_k}N^*\to \Lambda^{\ell_k-1}N^*\otimes \tOm_{\bP^N_{|X}}\to &\cdots& \to N^*\otimes S^{\ell_k-1}\tOm_{\bP^N_{|X}} \to S^{\ell_k}\tOm_{\bP^N_{|X}}\to S^{\ell_k}\tOm_X\to 0
\end{eqnarray*}
For $1\leqslant p\leqslant k$, we let
\begin{eqnarray*}
E_p^i&=&\Lambda^{i}N^*\otimes S^{\ell_p-i}\tOm_{\bP^N_{|X}} =\Lambda^{i}N^*\otimes S^{\ell_p-i}V\otimes \cO_X(i-\ell_p).
\end{eqnarray*}
By Lemma \ref{lemmedroite} to prove that 
$$H^j(X,S^{\ell_1}\tOm_X\otimes \cdots \otimes S^{\ell_k}\tOm_X\otimes \cO_{X}(a))=0 \ \ {\rm for} \ j< n-k\cdot c,$$
it suffices to prove that 
$$H^j(X,E_1^{i_1}\otimes \cdots \otimes E_k^{i_k}\otimes \cO_X(a))=0 \ \ {\rm for} \ j< n-k\cdot c+i_1+\cdots +i_k.$$
We now prove this fact. We have: 
\begin{eqnarray*}
H^j\!\!\!\!\!\!\!&(&\!\!\!\!\!\!\!X,E_1^{i_1}\otimes \cdots \otimes E_k^{i_k}\otimes \cO_{X}(a))=H^j(X,\Lambda^{i_1}N^*\otimes S^{\ell_1-i_1}\tOm_{\bP^N_{|X}}\otimes \cdots \otimes \Lambda^{i_k}N^*\otimes S^{\ell_k-i}\tOm_{\bP^N_{|X}}\otimes \cO_{X}(a))\\
&=&H^j(X,\Lambda^{i_1}N^*\otimes \cdots \otimes \Lambda^{i_k}N^*\otimes S^{\ell_1-i_1}V\otimes \cdots \otimes S^{\ell_k-i_k}V\otimes \cO_X(i_1-\ell_1+\cdots +i_k-\ell_k+a))\\
&=&  S^{\ell_1-i_1}V\otimes \cdots \otimes S^{\ell_k-i_k}V\otimes H^j(X,\Lambda^{i_1}(N^*(1))\otimes \cdots \otimes \Lambda^{i_k}(N^*(1))\otimes\cO_X(a-\ell_1-\cdots -\ell_k))
\end{eqnarray*}
On the other hand using Serre duality,  we obtain
\begin{eqnarray*}
H^j\!\!\!\!\!\!\!&(&\!\!\!\!\!\!\! X,\Lambda^{i_1}(N^*(1))\otimes \cdots \otimes \Lambda^{i_k}(N^*(1))\otimes\cO_X(a-\ell_1-\cdots -\ell_k)\\
&=& H^{n-j}(X,\Lambda^{i_1}(N(-1))\otimes \cdots \otimes \Lambda^{i_k}(N(-1))\otimes\cO_X(\ell_1+\cdots +\ell_k-a)\otimes K_X)
\end{eqnarray*}
Since $\cO_X(\ell_1+\cdots +\ell_k-a)$ is ample and $N(-1)$ is a rank-$c$ globally generated vector bundle, we can apply  Theorem \ref{Manivel} to see that this last cohomology group vanishes if
$$n-j>k\cdot c -i_1-\cdots -i_k$$
or equivalently, if
$$j<n-k\cdot c +i_1+\cdots +i_k.$$
This is exactly what was needed to conclude the proof.
\end{proof}
We are in a position to deduce the announced generalization of Schneider's theorem.

\begin{theorem}\label{SchneiderOm}
Let $X\subseteq \bP^N$ be a smooth variety of dimension $n$ and codimension $c=N-n$. Consider an integer $k\geqslant 1$ and $k$ integers $\ell_1,\dots, \ell_k\geqslant 0$, and $a\in \mathbb{Z}$. If  $j< n-k\cdot c$ and $a<\ell_1+\cdots +\ell_k-\min\{j,k\}$, then 
$$H^j(X,S^{\ell_1}\Om_X\otimes \cdots \otimes S^{\ell_k}\Om_X\otimes \cO_{X}(a))=0.$$
\end{theorem}

\begin{remark}
The particular case $k=1$ is precisely Schneider's result.
\end{remark}

\begin{proof}
Consider the $k$ short exact sequences gotten from the Euler exact sequence
\begin{eqnarray*}
0\to S^{\ell_1}\Om_X \to &S^{\ell_1}\tOm_X&\to S^{\ell_1-1}\tOm_X\to 0\\ 
0\to S^{\ell_2}\Om_X \to &S^{\ell_2}\tOm_X& \to S^{\ell_2-1}\tOm_X\to 0 \\
&\vdots &\\
0\to S^{\ell_k}\Om_X \to &S^{\ell_k}\tOm_X& \to S^{\ell_k-1}\tOm_X\to 0.
\end{eqnarray*}
Fix $j<n-k\cdot c$. By Lemma \ref{lemmegauche} it is now sufficient to check that for any $0\leqslant i_1,\dots,i_k \leqslant 1$, for all $0\leqslant i\leqslant j-i_1-\dots - i_k$  and for all $a< \ell_1+\cdots +\ell_k-\min\{j,k\}$
$$H^i(X,S^{\ell_1-i_1}\tOm_X\otimes \cdots \otimes S^{\ell_k-i_k}\tOm_X\otimes \cO_{X}(a))=0.$$
 This follows from Theorem $\ref{SchneidertOm}$ as soon as one observes that under these conditions, $i_1+\dots +i_k\leqslant \min\{j,k\}$. 

\end{proof}
 
\section{Application to Green-Griffiths  jet differentials}

Now we apply those vanishing results to Green-Griffiths jet differential bundle. The idea of converting a vanishing result for symmetric differential form bundles into a vanishing result for Green-Griffiths jet differential bundles is due to Diverio. In $\cite{Div08}$, Diverio uses a vanishing theorem due to Bruckmann and Rackwitz (see $\cite{B-R90}$) concerning symmetric differential forms on complete intersection varieties. 
We proceed along the lines of \cite{Div08}.

We start by recalling a cohomological lemma (see \cite{Div08}).

\begin{lemma}\label{DiverioLemma}
Let $E\to X$ be a holomorphic vector bundle with a filtration $\{0\}= E_r\subset \cdots \subset E_1 \subset E_0=E$. If $H^q(X,{\Gr}^{\bullet}E)=0$, then $H^q(X,E)=0$.
\end{lemma}

Combining this lemma with Theorem \ref{SchneiderOm} one derives our result.

\begin{corollary}\label{VanishingGGjets}
Let $X\subseteq \bP^N$ be a smooth variety of dimension $n$ and codimension $c$. Let $a\in \mathbb{Z}$.
If $j< n-k\cdot c$ and $a<\frac{m}{k}-\min\{j,k\}$, then
$$H^j(X,E^{GG}_{k,m}\Om_X\otimes \cO_X(a))=0.$$
\end{corollary}

\begin{proof}
Recall that the Green-Griffiths bundle admits a filtration  whose graded bundle is given by
$${\Gr}^{\bullet}E^{GG}_{k,m}\Om_X=\bigoplus_{\ell_1+2\ell_2+\cdots +k\ell_k=m}S^{\ell_1}\Om_X\otimes S^{\ell_2}\Om_X \otimes \cdots \otimes S^{\ell_k}\Om_X.$$
By twisting everything by $\cO_X(a)$, we obtain a filtration for $E^{GG}_{k,m}\Om_X\otimes \cO_X(a)$ whose graded bundle is given by 
$${\Gr}^{\bullet}E^{GG}_{k,m}\Om_X\otimes \cO_X(a)=\bigoplus_{\ell_1+2\ell_2+\cdots +k\ell_k=m}S^{\ell_1}\Om_X\otimes S^{\ell_2}\Om_X \otimes \cdots \otimes S^{\ell_k}\Om_X\otimes \cO_X(a).$$
 By Lemma $\ref{DiverioLemma}$, we only have to prove that 
$$H^j(X,S^{\ell_1}\Om_X\otimes S^{\ell_2}\Om_X \otimes \cdots \otimes S^{\ell_k}\Om_X\otimes \cO_X(a))=0$$
for all $j< n-k\cdot c$ and $\ell_1+\cdots +k\ell_k=m$. 
The hypothesis then gives us
$$a+\min\{j,k\}<\frac{m}{k}=\frac{\ell_1+\cdots +k\ell_k}{k}\leqslant \frac{k\ell_1+\cdots +k\ell_k}{k}=\ell_1+\cdots+\ell_k.$$
Thus we can conclude by applying Theorem $\ref{SchneiderOm}$.
\end{proof}

\begin{remark}
When $k=1$, this is exactly Schneider's theorem. When $X$ is a complete intersection, $a=0$ and $j=0$ then this is precisely Diverio's theorem.
\end{remark}

\begin{remark}\label{rknonvanishing}
Assertion 2 in Corollary \ref{VanishingGGjets} is optimal in $k$. Let $X$ be a complete intersection variety in $\bP^N$ of dimension $n$ and codimension $c=N-n$, denotes its multidegree by $(d_1,\dots, d_c)$. It is possible to prove that for any $k\geqslant \frac{n}{c}$, if the $d_i$ are big enough, then $0\neq H^0(X,E^{GG}_{k,m}\Om_X\otimes \cO_X(a))$ for  $m\gg 0$. In fact one obtains a stronger result: under the same hypothesis, it is possible to prove that  $0\neq H^0(X,E_{k,m}\Om_X\otimes \cO_X(a))$, where $E_{k,m}\Om_X\subseteq E_{k,m}^{GG}\Om_X$ denotes the Demailly-Semple jet differential bundle. We refer to $\cite{Bro11}$ for a proof.
\end{remark}



\section{Application to Pacienza-Rousseau's generalized Green-Griffiths jet differentials}

In \cite{P-R08}, Pacienza and Rousseau generalized Green-Griffiths jet differential bundles, $E^{GG}_{p,k,m}$. And, among other things, they generalized Diverio's vanishing theorem to those bundles.

\begin{theorem}[Pacienza-Rousseau \cite{P-R08} Theorem 5.1]\label{P-RVanishing}
Let $X\subseteq \bP^N$ ba a smooth complete intersection. Then

$$H^0(X,E_{p,k,m}^{GG})=0,$$

for all $m\geqslant 1$ and $k$ such that 
$$\binom{k+p}{p}-1<\frac{\dim (X)}{\codim (X)}.$$
\end{theorem}

There proof rests on the same idea than the proof of Diverio's Theorem \ref{DiverioVanishing}. Therefore it is not surprising that one can prove a more general statement using Theorem \ref{SchneiderOm}.

\begin{corollary}\label{VanishingPRGGjets}
Let $X\subseteq \bP^N$ be a smooth variety of dimension $n$ and codimension $c$. Let $a\in \mathbb{Z}$.
If $0\leqslant j< N-\binom{k+p}{p}\cdot c$ and $ a< \frac{m}{k} -\min\left\{j,\binom{k+p}{p}-1\right\}$, then
$$H^j(X,E_{p,k,m}^{GG}\otimes \cO_X(a))=0.$$
\end{corollary}

\begin{proof}
Recall (\cite{P-R08} Remark $2.6$) that there is a filtration for $E^{GG}_{p,k,m}\Om_X\otimes \cO_X(a)$ such that the graded terms are of the form 

\begin{eqnarray}
\bigotimes_{\alpha \in I_1}S^{q_{\alpha}^1}\Om_X \cdots \bigotimes_{\alpha \in I_{k}} S^{q_{\alpha}^k}\Om_X\otimes \cO_X(a),\label{PRgraduation}  
\end{eqnarray} 
where $I_{\ell}:=\{\alpha=(\alpha_1,\dots, \alpha_p)\  /\ \alpha_1+\cdots + \alpha_p=\ell\}$ and where 
$\sum_{\ell=1}^k\sum_{\alpha\in I_{\ell}} q_{\alpha}^{\ell}\alpha = (m,\dots , m).$\\
First observe that, 
$|I_1|+\cdots + |I_k|=\binom{k+p}{p}-1.$ 
Therefore there are only $\binom{k+p}{p}-1$ symmetric powers in expression ($\ref{PRgraduation}$). We apply Theorem $\ref{SchneiderOm}$, with the above observation, we just have to prove that under our hypothesis 
$$a<\sum_{\ell=1}^k\sum_{\alpha\in I_{\ell}} q_{\alpha}^{\ell} -\min\left\{j,\binom{k+p}{p}-1\right\}.$$
But this follows at once from 
\begin{eqnarray*}
\frac{(m,\dots , m)}{k}&=&\sum_{\ell=1}^k\sum_{\alpha\in I_{\ell}} q_{\alpha}^{\ell}\frac{\alpha}{k}\leq \sum_{\ell=1}^k\sum_{\alpha\in I_{\ell}} q_{\alpha}^{\ell}\frac{(k,\dots,k)}{k}=\left(\sum_{\ell=1}^k\sum_{\alpha\in I_{\ell}} q_{\alpha}^{\ell},\dots,\sum_{\ell=1}^k\sum_{\alpha\in I_{\ell}} q_{\alpha}^{\ell}\right).
\end{eqnarray*}
\end{proof}

\begin{remark}
Recall Hartshorne's conjecture for complete intersection variety (see \cite{Har74}): {\it if $X$ is a nonsingular subvariety of dimension $n$ of $\bP^N$, and if $n>\frac{2}{3}N$, then $X$ is a complete intersection.}
In view of this conjecture, our vanishing results are not surprising, at least in low codimension. It just shows that one can not use naively Diverio's vanishing theorem nor Pacienza-Rousseau's vanishing theorem to distinguish  a complete intersection variety from a variety which is not complete intersection.
\end{remark}


\section{A further generalization}

One can generalize Theorem \ref{SchneidertOm} and \ref{SchneiderOm} even further if one uses the whole strength of Manivel's results.

\begin{theorem}[Manivel \cite{Man97} Theorem A]
Let $E$ be a holomorphic vector bundle of rank $e$, and $L$ a line bundle on a smooth projective complex variety $X$ of dimension $n$. Suppose that $E$ is ample and $L$ is nef, or that $E$ is nef and $L$ ample. Then, for any sequences of integers $k_1, \dots, k_{\ell}$ and $j_1,\dots , j_m$,
$$H^{p,q}(X,S^{k_1}E\otimes \cdots \otimes S^{k_{\ell}}E\otimes \Lambda^{j_1}E\otimes \cdots \otimes \Lambda^{j_m}E\otimes (\det E)^{\ell+n-p}\otimes L)=0$$
as soon as $p+q>n+\sum_{s=1}^m(e-j_s).$
\end{theorem}

Redoing the proof of Theorem \ref{SchneidertOm} and \ref{SchneiderOm}, and considering moreover exterior powers of $\tOm_X$ and $\Om_X$ one obtains the following.

\begin{theorem}
Let $X\subseteq \bP^N$ be a smooth variety of dimension $n$ and codimension $c=N-n$. Consider integers $j,k,p,r\geq 0$, integers $\ell_1,\dots, \ell_k, m_1,\dots, m_r\geq 0$, and $a,q\in \mathbb{Z}$. Let $\ell=\sum_{i=1}^k\ell_i$ and $m=\sum_{i=1}^rm_i$. 
\begin{enumerate}
\item If $j+p< n-kc$ and $K_X^{-q-p-r}\otimes \cO_{X}(\ell+m-(n+1)(p+r)-a)$ is ample, then 
$$H^{p,j}(X,S^{\ell_1}\tOm_X\otimes \cdots \otimes S^{\ell_k}\tOm_X\otimes \Lambda^{m_1}\tOm_X\otimes\cdots \otimes \Lambda^{m_r}\tOm_X \otimes K_X^{q} \otimes \cO_{X}(a))=0.$$
\item If $j+p< n-kc$ and $K_X^{-q-p-r}\otimes \cO_{X}(\ell+m-(n+1)(p+r)-a-\min\{j,k+m\})$ is ample, then
$$H^{p,j}(X,S^{\ell_1}\Om_X\otimes \cdots \otimes S^{\ell_k}\Om_X\otimes \Lambda^{m_1}\Om_X\otimes\cdots \otimes \Lambda^{m_r}\Om_X \otimes K_X^{q} \otimes \cO_{X}(a))=0.$$
\end{enumerate}
\end{theorem}

\emph{Acknowledgment.} We would like thank C. Mourougane for the discussions we had concerning this work and for the time he spent reading and commenting earlier versions of this paper.

\bibliographystyle{alpha}      
\bibliography{biblio.bib}   

\begin{thebibliography}{{Mer}10}

\bibitem[BDO08]{BdO08}
Fedor Bogomolov and Bruno De~Oliveira.
\newblock Symmetric tensors and geometry of {$\Bbb P^N$} subvarieties.
\newblock {\em Geom. Funct. Anal.}, 18(3):637--656, 2008.

\bibitem[BR90]{B-R90}
P.~Br{\"u}ckmann and H.-G. Rackwitz.
\newblock {$T$}-symmetrical tensor forms on complete intersections.
\newblock {\em Math. Ann.}, 288(4):627--635, 1990.

\bibitem[{Bro}11]{Bro11}
D.~{Brotbek}.
\newblock {Hyperbolicity Related Problems for Complete Intersection Varieties}.
\newblock {\em ArXiv e-prints}, January 2011.

\bibitem[Deb05]{Deb05}
Olivier Debarre.
\newblock Varieties with ample cotangent bundle.
\newblock {\em Compos. Math.}, 141(6):1445--1459, 2005.

\bibitem[Dem97]{Dem97}
Jean-Pierre Demailly.
\newblock Algebraic criteria for {K}obayashi hyperbolic projective varieties
  and jet differentials.
\newblock In {\em Algebraic geometry---{S}anta {C}ruz 1995}, volume~62 of {\em
  Proc. Sympos. Pure Math.}, pages 285--360. Amer. Math. Soc., Providence, RI,
  1997.

\bibitem[Div08]{Div08}
Simone Diverio.
\newblock Differential equations on complex projective hypersurfaces of low
  dimension.
\newblock {\em Compos. Math.}, 144(4):920--932, 2008.

\bibitem[EL93]{E-L93}
Lawrence Ein and Robert Lazarsfeld.
\newblock Syzygies and {K}oszul cohomology of smooth projective varieties of
  arbitrary dimension.
\newblock {\em Invent. Math.}, 111(1):51--67, 1993.

\bibitem[Har74]{Har74}
Robin Hartshorne.
\newblock Varieties of small codimension in projective space.
\newblock {\em Bull. Amer. Math. Soc.}, 80:1017--1032, 1974.

\bibitem[LP75]{LeP75}
J.~Le~Potier.
\newblock Annulation de la cohomolgie \`a valeurs dans un fibr\'e vectoriel
  holomorphe positif de rang quelconque.
\newblock {\em Math. Ann.}, 218(1):35--53, 1975.

\bibitem[Man97]{Man97}
Laurent Manivel.
\newblock Vanishing theorems for ample vector bundles.
\newblock {\em Invent. Math.}, 127(2):401--416, 1997.

\bibitem[{Mer}10]{Mer10}
J.~{Merker}.
\newblock {Complex projective hypersurfaces of general type: toward a
  conjecture of Green and Griffiths}.
\newblock {\em ArXiv e-prints}, May 2010.

\bibitem[PR08]{P-R08}
G.~{Pacienza} and E.~{Rousseau}.
\newblock {Generalized Demailly-Semple jet bundles and holomorphic mappings
  into complex manifolds}.
\newblock {\em ArXiv e-prints}, October 2008.

\bibitem[Sch92]{Sch92}
Michael Schneider.
\newblock Symmetric differential forms as embedding obstructions and vanishing
  theorems.
\newblock {\em J. Algebraic Geom.}, 1(2):175--181, 1992.

\end{thebibliography}

\end{document}